\documentclass{amsproc}
\usepackage{euscript}
\usepackage{cases}
\usepackage{mathrsfs}
\usepackage{bbm}
\usepackage{amssymb}
\usepackage{amsfonts,amsmath,amsxtra,mathdots,mathabx}
\usepackage{color}
\usepackage{hyperref}
\usepackage{tikz}
\usepackage{appendix,upgreek}

\allowdisplaybreaks

\DeclareFontFamily{U}{matha}{\hyphenchar\font45}
\DeclareFontShape{U}{matha}{m}{n}{
	<5> <6> <7> <8> <9> <10> gen * matha
	<10.95> matha10 <12> <14.4> <17.28> <20.74> <24.88> matha12
}{}
\DeclareSymbolFont{matha}{U}{matha}{m}{n}

\DeclareMathSymbol{\Lt}{3}{matha}{"CE}
\DeclareMathSymbol{\Gt}{3}{matha}{"CF}

\DeclareSymbolFont{mathc}{OML}{txmi}{m}{it}
\DeclareMathSymbol{\varuu}{\mathord}{mathc}{117}
\DeclareMathSymbol{\varvv}{\mathord}{mathc}{118}
\DeclareMathSymbol{\varww}{\mathord}{mathc}{119}

\def\valpha{\text{\scalebox{0.84}[1.02]{$\alpha$}}}   
\def\vepsilon{\upvarepsilon}
\def\vnu{\text{{\scalebox{0.9}[1]{$\nu$}}}}

\newcommand{\BR}{{\mathbf {R}}} 
\newcommand{\BZ}{{\mathbf {Z}}}

\newcommand{\GL}{{\mathrm {GL}}}

\newcommand{\SL}{{\mathrm {SL}}}

\newcommand{\ra}{\rightarrow} 
\def\sumx{\sideset{}{^\star}\sum}

\def\nd{\mathrm{d}}

\def\shskip{\hspace{0.5pt}}

\newcommand{\delete}[1]{}

\theoremstyle{plain}

\newtheorem{lem}{Lemma}[section]
\newtheorem{theorem}{Theorem}[section]

\newtheorem*{thm*}{Theorem}

\theoremstyle{remark} 
\newtheorem{remark}{Remark}[section] 
\newtheorem{defn}{Definition}[section]

\newtheorem*{acknowledgement}{Acknowledgements}

\numberwithin{equation}{section}

\begin{document}
	
		\title[$\mathrm{GL}_4 \! \times \! \mathrm{GL}_2$\,$L$-functions at Special Points]{{The Second Moment of  $\mathrm{GL}_4 \times \mathrm{GL}_2$   $L$-functions at Special Points}}

	\begin{abstract}
		In this paper, we provide an alternative proof of Chandee and Li's result on  the second moment of $ \mathrm{GL}_4 \times \mathrm{GL}_2$ special $L$-values. Our method is conceptually more direct as it neither detects the `Eisenstein--Kloosterman' cancellation nor uses the Poisson summation formula. 
	\end{abstract}
	
	\author[Z. Qi  and R. Qiao]{Zhi Qi and Ruihua Qiao}
	\address{School of Mathematical Sciences\\ Zhejiang University\\Hangzhou, 310027\\China}
	\email{zhi.qi@zju.edu.cn, ruihua.qiao@zju.edu.cn}
	
	\thanks{The first author was supported by National Key R\&D Program of China No. 2022YFA1005300.}

	
	\maketitle

	\section{Introduction}

	Let  $u_j $ traverse an orthonormal basis of  Hecke--Maass  forms for  $\mathrm{SL}_2 (\BZ) $. Let $s_j (1-s_j) = 1/4+t_j^2$, with  $s_j = 1/2+ i t_j$ ($t_j > 0$),  be the Laplace eigenvalue of $u_j $. It is proven by  Chandee and Li \cite{Chandee-Li-GL(4)-Special-Points} that
	\begin{align}\label{1eq: mean Lindelof, T}
		\sum_{t_j \leqslant T}    |L (s_j, \phi  \times u_j)  |^2 \Lt_{\phi, \vepsilon}   T^{2 +\vepsilon},
	\end{align}  
for any fixed  Hecke--Maass cusp form $\phi$ for $\SL_4 (\BZ)$.  Previously, \eqref{1eq: mean Lindelof, T} was proven for $\phi$ a fixed $ \SL_3 (\BZ)$ cusp form by Young \cite{Young-GL(3)-Special-Points}, who used a refined asymptotic large sieve in the spirit of Luo, Iwaniec, and Li \cite{Luo-Twisted-LS,Iwaniec-Li-Ortho}. 

The key point of Luo's large sieve is the `Eisenstein--Kloosterman' cancellation detected by the  Euler--Maclaurin formula, while Young's refinement of this is by the Poisson summation formula: part of the zero frequency is canceled with part of the Eisenstein contribution. 
The refined large sieve of Young is a main tool of Chandee and Li. Moreover, a major step in their work is an application of Poisson summation. 

Therefore Poisson summation has been applied twice in total by Young, Chandee, and Li!  However, since Poisson is an involution,  two applications of Poisson should become futility, if  the cancellation in the zero frequency were disregarded.   

Let us also observe that  the `Eisenstein--Kloosterman' cancellation does not really play a role in the problem, since the Eisenstein contribution is already $O (  T^{2+\vepsilon})$: 
\begin{align*}
	|L (1/2 , \phi) |^2 \cdot	\int_{- T}^{T }    |L (1/2+2it , \phi)   |^2  \nd t \Lt_{\phi, \vepsilon}     T^{2 +\vepsilon},
\end{align*}  
as a result of the large sieve for Dirichlet polynomials \cite[Theorem 6.1]{Montgomery-Topics} (with $N $ up to $ T^{2+\vepsilon}$),  
\begin{align*}
	\int_{-T}^T \bigg|\sum_{ n \leqslant   N} a_n n^{it} \bigg|^2 \nd t \Lt (T+N) \sum_{ n \leqslant   N}  |a_n|^2, 
\end{align*} 
so its part canceled out should not exceed $ T^{2+\vepsilon} $ any way.

The purpose of this paper is to provide a straightforward `Poisson-free' proof of \eqref{1eq: mean Lindelof, T} by the  twisted large sieve on short intervals established in \cite{Qi-GL(3)-Special-Points} (recorded below in Theorem \ref{thm: LS GL(2)}).  It was used to improve Young's $\mathrm{GL}_3 \times \GL_2$ result into
\begin{align}\label{1eq: mean Lindelof, GL(3)}
	\sum_{T < t_j \leqslant T + {T}^{1/2} }    |L (s_j, \phi  \times u_j)  |^2 \Lt_{\phi, \vepsilon}   T^{3/2 +\vepsilon} . 
\end{align}     
Our theorem is a `short-interval' equivalence of \eqref{1eq: mean Lindelof, T} as follows. 

\begin{theorem} \label{thm: main}
	Let $ \phi $ be a fixed  Hecke--Maass cusp form for $\SL_4 (\BZ)$. Then
	 \begin{align}\label{1eq: bound GL(4)}
	 	\sum_{T < t_j \leqslant T + {T}^{1 - \vepsilon}  }    |L (s_j, \phi  \times u_j)  |^2 \Lt_{\phi, \vepsilon}   T^{2 +\vepsilon} ,
	 \end{align}
 for any $\vepsilon > 0$, where the implied constant depends only  on $\phi$ and $\vepsilon$.  
\end{theorem}

Note that $(T, 2 T]$ splits into $O (T^{\vepsilon})$ many intervals of the form $(T_1, T_1 + T_1^{1-\vepsilon}]$, and hence \eqref{1eq: mean Lindelof, T} is indeed deducible from \eqref{1eq: bound GL(4)} by the $\vepsilon$-convention. As explained in Remark \ref{rem: short}, it is a technical reason that the interval length be (at most) $T^{1-\vepsilon}$ for the application of Theorem \ref{thm: LS GL(2)}.

\subsection*{Comparison with the work of Chandee and Li} 

To be explicit, the reader should first examine  the similarities between our \eqref{5eq: P(n; N)}, \eqref{5eq: P(N), 0} with (6.1) in \cite{Chandee-Li-GL(4)-Special-Points}; after the approximate functional equation, the former is obtained directly  in a `Poisson-free' way from  Theorem \ref{thm: LS GL(2)}, while the latter is obtained from the twisted spectral large sieve of Young \cite{Young-GL(3)-Special-Points}  and the application of Poisson summation. 

The bulk of analysis comes after the application of  $\GL_4$ Vorono\"i and the simplification of exponential sums. Let us outline the similarities  and differences between the analysis in \S \S \ref{sec: Hankel, 1}--\ref{sec: large sieve, 2} of this work and \S \S 8, 9 of \cite{Chandee-Li-GL(4)-Special-Points}. There are two cases, depending on whether the argument of the Bessel kernel in the Hankel integral transform is small or large. For the first case as in \S \S \ref{sec: Hankel, 1}, \ref{sec: large sieve, 1} and \S 8 of  \cite{Chandee-Li-GL(4)-Special-Points} (this case may be avoided in the $\GL_3$ setting in \cite{Qi-GL(3)-Special-Points}),  both proofs apply the Mellin technique and the classical large sieve, but the analysis here is in a less explicit form due to the notion of inert functions.  For the second case as in  \S \S \ref{sec: Hankel II}, \ref{sec: large sieve, 2} and \S 9 of  \cite{Chandee-Li-GL(4)-Special-Points}, 
 our proof is considerably shorter than theirs thanks to 
the hybrid large sieve of Young \cite{Young-GL(3)-Special-Points}  so that we can avoid opening the square (the referee pointed out to us that this alternative was mentioned at the end of \S 4 of \cite{Chandee-Li-GL(4)-Special-Points}).

Moreover, it might be of its own interest that an individual bound for $\GL_4$ Fourier coefficients is established in Lemma \ref{lem: bound for A}  to  improve slightly the averaged Ramanujan bounds in \cite{Chandee-Li-GL(4)-Special-Points} by a simple argument of Blomer \cite{Blomer}.

\subsection*{Notation}

By $X \Lt Y$ or $X = O (Y)$ we mean that $|X| \leqslant c Y$  for some constant $c  > 0$, and by $X \asymp Y$ we mean that $X \Lt Y$ and $Y \Lt X$. We write $X \Lt_{\valpha, \beta, ...} Y $ or $  X = O_{\valpha, \beta, ...} (Y) $ if the implied constant $c$ depends on $\valpha$, $\beta$, ....  

The notation $x \sim X$ stands for  $ X <  x \leqslant 2 X $ for $x$ integral or real according to the context.  

By `negligibly small' we mean $ O_A ( T^{-A} )$ for arbitrarily large but fixed $A > 0$. 

Throughout the paper,  $\vepsilon  $ is arbitrarily small and its value  may differ from one occurrence to another. 

\begin{acknowledgement}
	We gratefully acknowledge the referee for many suggestions that significantly improved the clarify of our exposition. 
\end{acknowledgement}

\section{The Notion of Inert Functions}

According to \cite{KPY-Stationary-Phase}, let us introduce the notion of inert functions. 

\begin{defn}\label{defn: inert}
	Let $\boldsymbol{I}  \subset \BR_+^{d}$ be a product of intervals {\rm(}not necessarily bounded{\rm)}.  For $X \geqslant 1$, we say a smooth function $\varww \in C^{\infty} (\boldsymbol{I})$ {\rm(}not necessarily compactly supported{\rm)} is $X$-inert if 
	\begin{align*}
		\boldsymbol{x}^{\boldsymbol{j}} 	\varww^{(\boldsymbol{j})} (\boldsymbol{x})  \Lt_{\boldsymbol{j}} X^{|\boldsymbol{j}|}  , \qquad \text{($\boldsymbol{x} \in \boldsymbol{I}$),} 
	\end{align*} for every $\boldsymbol{j} \in \mathbf{N}_0^{d}$, where in the multi-variable notation $\boldsymbol{x}^{\boldsymbol{j}} = x_1^{j_1} \cdots x_d^{j_d}$, $ \varww^{(\boldsymbol{j})} (\boldsymbol{x}) = \varww^{(j_1, \cdots, j_d)} (x_1, \cdots, x_d) $, and $|\boldsymbol{j}| = j_1+ \cdots + j_{d}$. 
\end{defn}

For $\boldsymbol{I}$ bounded open and $ \varww \in C_c^{\infty} (\boldsymbol{I}) $ compactly supported,  let us define the Mellin transform $ \widetilde{\varww} (\boldsymbol{s}) $ by 
	\begin{align*}
		\widetilde{\varww} (\boldsymbol{s}) = \int_{ \boldsymbol{I} } \varww (\boldsymbol{x}) \boldsymbol{x}^{\boldsymbol{s}   }  \nd^{\times} \boldsymbol{x}, 
	\end{align*}
where as usual $ \nd^{\times} x = \nd x / x $ and $ \nd^{\times} \boldsymbol{x} = \nd^{\times} x_1  \cdots \nd^{\times} x_d $. 
	We have the Mellin inversion formula
	\begin{align*}
		\varww (\boldsymbol{x}) = \frac 1 {(2\pi)^d } \int_{ \BR^{d} } \widetilde{\varww} (i \boldsymbol{r})  \boldsymbol{x}^{- i \boldsymbol{r}} \nd \boldsymbol{r}. 
	\end{align*}	In this paper, we shall  have practically $X = \log T$ or $T^{\vepsilon}$, and 
\begin{align*}
	\int_{ \boldsymbol{I} }    { \nd^{\times} \boldsymbol{x}} = O (\log^d T), 
\end{align*}
so  it is routine to prove that $\widetilde{\varww} (i \boldsymbol{r})$ is trivially $O (\log T)$ and negligibly small except when each component of $\boldsymbol{r}$ is bounded by $T^{\vepsilon}$. Accordingly, define 
\begin{align*}
	\boldsymbol{I}^d (X) = \big\{ \boldsymbol{r} \in \BR^d : |r_1|, \cdots, |r_d| \leqslant X \big\},
\end{align*}
then
\begin{align}\label{2eq: Mellin}
	\varww (\boldsymbol{x}) = \frac 1 {(2\pi)^d } \int_{\boldsymbol{I}^d (T^{\vepsilon}  )}  \widetilde{\varww} (i \boldsymbol{r})  \boldsymbol{x}^{- i \boldsymbol{r}} \nd \boldsymbol{r} + O_A \big(T^{-A} \big). 
\end{align} 
This will be used {\rm(}together with Cauchy--Schwarz{\rm)} to separate variables prior to the application of large sieve. 

\begin{remark}\label{rem: Mellin} We remark that $ \varww \in C^{\infty} (\boldsymbol{I}) $ actually does not need to be compactly supported if we are concerned with $\varww (\boldsymbol{x})$ for $\boldsymbol{x}$ restricted to a certain closed $\boldsymbol{I}_{0} \subset \boldsymbol{I}$ (for example, $[1, T^{\theta}] \subset (1/2, 2T^{\theta})$) such that there is an inert cut-off function $\varvv  \in C_c^{\infty} (\boldsymbol{I}) $ with $ \varvv (\boldsymbol{x}) \equiv 1 $ on $\boldsymbol{I}_0$. Of course, \eqref{2eq: Mellin} needs to be modified with $ \widetilde{\varww} (i \boldsymbol{r}) $ replaced by $ \widetilde{\varww}_{0} (i \boldsymbol{r})  $ for $ \varww_0  = \varww \cdot \varvv  $, but we shall still write $ \widetilde{\varww} (i \boldsymbol{r}) $ with slight abuse of notation.  
\end{remark}

\section{A Stationary Phase Lemma}

Let us record here Lemma 2.5 from \cite{Qi-GL(3)-Special-Points}.

\begin{lem} \label{lem: analysis of integral}
	Let  $ \gamma > 1$. 
	For $ \sqrt {\lambda} \geqslant X \geqslant  1$ and $\rho > 0$, define 
	\begin{align*}
		I_{\gamma}^{\pm} (\lambda, \boldsymbol{x}) =   \int_{\rho}^{2\rho}  e \big(\lambda \big(x \pm \gamma   x^{1/\gamma} \big) \big) \varww (x, \lambda, \boldsymbol{x}) \nd x,  
	\end{align*} 
	for an $X$-inert function  $\varww (x, \lambda, \boldsymbol{x}) \in C^{\infty} ([\rho ,  2 \rho] \times [X^2, \infty) \times \boldsymbol{I} )$, with compact support in the first variable $x$. 
	
	{\rm\,(i)} We have
	$$ I_{\gamma}^{\pm} (\lambda, \boldsymbol{x}) \Lt_A  \rho \cdot \bigg( \frac {  X  } {\lambda (\rho +\rho^{1/\gamma})}\bigg)^A  $$ 
	for any value of $\rho$ in the $+$ case, or for $ \min \big\{ \rho/\sqrt{2}, \sqrt{2}/\rho \big\}   < 1 / 2 $ in the $-$ case.  
	
	{\rm(ii)} Define  
	\begin{align*}
		\varvv_{\gamma} (\lambda, \boldsymbol{x} ) =  e (  \lambda (\gamma -1) ) \cdot \sqrt{\lambda}   I_{\gamma}^{-} (\lambda, \boldsymbol{x}  ), 
	\end{align*}  
	then $\varvv_{\gamma} (\lambda, \boldsymbol{x} )$ is an $X$-inert function for any $1/2 \leqslant \rho/\sqrt{2} \leqslant 2 $. 
\end{lem}

\section{The Large Sieve Inequalities}

\subsection{The Classical Large Sieve}  

First we have the classical large sieve inequality (see \cite[(3.12)]{Montgomery-Topics}).
\begin{lem} \label{lem: LS} Let $C, N > 0$. Then 
	\begin{align*}
	 \sum_{c \shskip \leqslant C} \,  \sumx_{   a            (\mathrm{mod} \, c) }   \bigg| \sum_{M < n \leqslant M+N}  a_{n}   e \Big(   \frac {a           n} {c} \Big)  \bigg|^2 \Lt \big(C^2  + N   \big) \sum_{M < n \leqslant M+N}  |a_{n}  |^2 , 
\end{align*} 
	for any complex $a_n$. 
\end{lem}

\subsection{The Hybrid Large Sieve of Young} 

Next, we record here \cite[Lemma 2.6]{Qi-GL(3)-Special-Points} which is a special case of the hybrid large sieve of Young \cite[Lemma 6.1]{Young-GL(3)-Special-Points}.  

\begin{lem}\label{lem: Young's LS} Let   $v, \tau, C, N  > 0$ and $\gamma \neq 0$ be real. Then
	\begin{align*}
		\begin{aligned}
			\int_{-\tau}^{\tau}   \sum_{c \shskip \leqslant C}  \frac 1 {c} \, \sumx_{   a            (\mathrm{mod} \, c) } \!  \bigg|  \!   \sum_{ n  \sim N}    a_{n}   e \Big(   \frac {a           n} {c} \Big)  e  \bigg(   \frac { n^{\gamma} t} {c v} \bigg)   \bigg|^2  \!   \nd t \Lt_{\gamma} \! \big(\tau C + v N^{1-\gamma} \log C \big) \!   \sum_{ n \sim N }    |a_n|^2,
		\end{aligned}
	\end{align*}
	for any complex $a_n$.  
\end{lem}

\subsection{The Special Twisted Spectral Large Sieve for  {\protect\scalebox{1.06}{$\SL_2 (\BZ)$}}}

Our main tool is the following spectral large sieve  for $\SL_2 (\BZ)$ essentially from \cite[Theorem 2]{Qi-GL(3)-Special-Points}. It was derived from the Kuznetsov trace formula and a certain representation of the Bessel integral. 
For simplicity, let us reformulate it as follows.

\begin{theorem} \label{thm: LS GL(2)}
	Let $\vepsilon > 0$. Assume $ T^{\vepsilon} \leqslant M \leqslant T^{1-\vepsilon}$.  Define 
	\begin{align*}
		\breve{S} (  \mathcal{A})   =	   \sum_{T < t_j \leqslant T+M}    \bigg| \sum_{   n \sim N }  a_{n} \lambda_{j} (n) n^{i t_j}  \bigg|^2, 
	\end{align*}
for any complex sequence  $\mathcal{A} = \{a_n\}$. 
Then 
\begin{align*}
	\breve{S} (  \mathcal{A}) \Lt  T^{\vepsilon} \big( \breve{D} (  \mathcal{A}) + \breve{E}  (  \mathcal{A}) + \breve{P}  (  \mathcal{A})   \big),
\end{align*}
where 
\begin{align*}
	\breve{D} (  \mathcal{A})  =  M T \sum_{ n \sim N } |a_{n}|^2, \qquad \breve{E}  ( \mathcal{A}) = \frac {N^{3/2+\vepsilon}} {T^A}   \sum_{ n \sim N } |a_{n}|^2, 
\end{align*}
for any $A > 0$, and 
\begin{align*}
	\breve{P}  ( \mathcal{A}) = M T \sum_{q \Lt N/T}  \frac 1 {q} \int_{-M^{\vepsilon}/ M}^{M^{\vepsilon}/ M} \sum_{c \Lt N/ T q} \frac 1 {c}  \, \sumx_{   a            (\mathrm{mod} \, c) } \bigg|   \sum_{ n \sim N }  a_{n}   e \Big(   \frac { {a      }     n} {c} \Big)  e \bigg( \frac {n t} {c q} \bigg) \bigg|^2  \nd t . 
\end{align*}
\end{theorem}

\begin{proof}
	As it is deduced from the Kuznetsov trace formula,  \cite[Theorem 2]{Qi-GL(3)-Special-Points} is more explicit, involving the harmonic weight $\omega_j$ and the smooth weight 
	 \begin{align*}
	 	h ( t  ) = \exp \bigg( \hspace{-1pt}  -  \frac {(t - T)^2} {M^2}    \bigg) + \exp \bigg(  \hspace{-1pt} -  \frac {(t + T)^2} {M^2}    \bigg) . 
	 \end{align*}
 More explicitly, we have the identity
 \begin{align*}
 	S (\mathcal{A})  + T (\mathcal{A})  = D (\mathcal{A}) + P (\mathcal{A}) , 
 \end{align*}
where 
\begin{align*}
		S (\mathcal{A})   & =	   \sum_{j = 1 }^{\infty}   \omega_j { h  ( t_j ) }   \bigg| \sum_{ n \sim N }  a_{n} \lambda_{j} (n) n^{i t_j}  \bigg|^2   ,   
\end{align*}
$T (\mathcal{A})$, $D (\mathcal{A})$, and $P (\mathcal{A})$ are the Eisenstein, diagonal, and off-diagonal (Kloosterman) contributions, respectively. It is proven in \cite[Theorem 2]{Qi-GL(3)-Special-Points} that 
\begin{align*}
	D (\mathcal{A}) = \frac {2} {\pi \sqrt{\pi}} \breve{D} (\mathcal{A}) + O \big(\breve{E} (\mathcal{A}) \big), \qquad P (\mathcal{A}) = O \big( \breve{P} (\mathcal{A}) \big), 
\end{align*}
and hence, if we drop the Eisenstein contribution $T (\mathcal{A})$ by positivity, then 
\begin{align*}
	S (  \mathcal{A}) \Lt \breve{D} (  \mathcal{A}) + \breve{E}  (  \mathcal{A}) + \breve{P}  (  \mathcal{A}) . 
\end{align*}
Further, we may replace $ S  (  \mathcal{A})$ by $ \breve{S} (  \mathcal{A}) $ at the cost of the factor $T^{\vepsilon}$, since for any $t_j$ on $(T, T+M]$ we have lower bounds $ h (t_j) \Gt 1 $ and $\omega_j \Gt T^{-\vepsilon}$, the latter due to Iwaniec  \cite[Theorem 2]{Iwaniec-L(1)}. Finally, we may as well drop the real-valued assumption on the sequence $\mathcal{A}$ in \cite[Theorem 2]{Qi-GL(3)-Special-Points} (as explained in \cite[Remark 1.1]{Qi-GL(3)-Special-Points}). 
\end{proof}

\begin{remark}\label{rem: short}
	It should be stressed that $ M \leqslant T^{1-\vepsilon}$ is crucial in the analysis in \cite[\S 3]{Qi-GL(3)-Special-Points}, so this explains the range $ T < t_j \leqslant T + {T}^{1 - \vepsilon} $ in Theorem \ref{thm: main}. 
\end{remark}

\section{\texorpdfstring{Maass Cusp Forms for {\protect\scalebox{1.06}{$\SL_4 (\BZ)$}}}{Maass Cusp Forms for SL\unichar{"2084}(Z)} }

Let $\phi$ be a Hecke--Maass cusp  form for $\SL_4 (\BZ)$ of Fourier coefficients $ A (n_1 , n_2    , n_3    )  $, normalized so that $ A(1,1,1) = 1 $, and Langlands parameters $   \{ \lambda    _1, \lambda    _2, \lambda    _3 ,\lambda    _4\} $, with $\lambda    _1+\lambda    _2+\lambda    _3+\lambda    _4 = 0$.  

\subsection{Averaged Ramanujan Bounds for {\protect\scalebox{1.06}{$\SL_4 (\BZ)$}}}


In this sub-section, we consider bounds for certain sums of the Fourier coefficients $A (n_1, n_2, n_3)$.  Clearly such bounds also hold  for the dual Fourier coefficients $ A (n_3, n_2, n_1) = \overline{ A (n_1, n_2, n_3)} $.   

Firstly, let us record here Lemma 2.2 from \cite{Chandee-Li-GL(4)-Special-Points}, whose proof relies not only on the theory of Rankin--Selberg, but also on the functoriality of  exterior square on $\GL_4$ due to Kim \cite{Kim-Sarnak}.   

\begin{lem}\label{lem: Ramanujan, 1}
	 We have
	 \begin{align*}
	 	\sum_{n_1 \leqslant X_1} \sum_{n_2 \leqslant X_2} |A(n_1, n_2, 1) |^2 \Lt_{\phi, \vepsilon} (X_1 X_2)^{1+\vepsilon}.   
	 \end{align*}
\end{lem}

Next, we shall prove in Appendix \ref{append} the following individual bound by the works of Kim--Sarnak \cite[Appendix 2]{Kim-Sarnak} and  Luo--Rudnick--Sarnak \cite{Luo-Rudnick-Sarnak}.  

\begin{lem}\label{lem: bound for A} Let  $$\theta_4=\frac 1 2 - \frac 1 {11}, \qquad \theta_6=\frac 1 2 -\frac 1 {37}, $$ be the exponents towards the Ramanujan conjecture for $ \GL_4 $ and $\GL_6$ in  \cite[Appendix 2]{Kim-Sarnak} and   \cite{Luo-Rudnick-Sarnak}. Then 
	 \begin{align*}
	 	A(n_1, n_2, n_3)   \Lt_{\vepsilon} n_1^{\theta_4+\vepsilon} n_2^{\theta_6 +\vepsilon} n_3^{\theta_4+\vepsilon} . 
	 \end{align*}
\end{lem}

As a corollary, we have an improvement of Lemma 3.5 in \cite{Chandee-Li-GL(4)-Special-Points} by adopting a simple argument of Blomer \cite{Blomer}.

\begin{lem}\label{lem: Ramanujan, 2}
	 We have
	 \begin{align*}
	 	\sum_{n  \leqslant X } |A(n, a_2, a_3) |^2 \Lt_{\phi, \vepsilon}    a_2^{35/37+\vepsilon} a_3^{9/11+\vepsilon} X.   
	 \end{align*} 
\end{lem}

\begin{proof}
	 Similar to the proof of (10) in \cite{Blomer},  by positivity and multiplicativity, 
	 \begin{align*}
	 	\sum_{n \leqslant X} |A(n , a_2, a_3) |^2 & \leqslant  
	 	\sum_{m|(a_2 a_3)^{\infty}}\mathop{\sum_{n\leqslant X /m}}_{(n,ma_2 a_3 )=1}|A(mn,a_2 ,a_3 )|^2\\
	 	& \leqslant  \sum_{m|(a_2 a_3 )^{\infty}} |A(m,a_2 ,a_3 )|^2 \sum_{n\leqslant X /m}|A(n,1,1)|^2 . 
	 \end{align*}
 Then the proof is completed by the Rankin--Selberg bound 
 \begin{align*}
 	 \sum_{n\leqslant X }|A(n,1,1)|^2 \Lt_{\phi} X, 
 \end{align*}
followed by the bound in Lemma \ref{lem: bound for A}. 
\end{proof}

\begin{remark}
	Similarly, we may also improve Lemma {\rm3.1} in \cite{Chandee-Li-GL(4)-Special-Points} {\rm(}although not required in our proof of Theorem {\rm\ref{thm: main}}{\rm)}, with the square-free condition on $a_1, a_3$ removed,  
\begin{align*}
	\sum_{n  \leqslant X } |A(a_1, n, a_3) |^2 \Lt_{\phi, \vepsilon}    (a_1 a_3)^{9/11+\vepsilon} X .   
\end{align*} 
This is a consequence of  Lemma {\rm\ref{lem: bound for A}} above and Lemma {\rm3.2} in \cite{Chandee-Li-GL(4)-Special-Points} {\rm(}the Rankin--Selberg bound for the exterior square $L$-function $L(s, \phi, \Lambda^2)${\rm)}
 \begin{align*}
	\sum_{n\leqslant X }|A(1,n,1)|^2 \Lt_{\phi} X . 
\end{align*}
\end{remark}


\subsection{Vorono\"i Summation Formula for {\protect\scalebox{1.06}{$\SL_4 (\BZ)$}}} 

The Vorono\"i summation formula for $\SL_4 (\BZ)$ is another main ingredient in our analysis.  In this sub-section, let us specialize the Vorono\"i summation formula in \cite{MZ-Voronoi}\footnote{In comparison to \cite{Miller-Schmid-2009}, the Vorono\"i summation formula for $\GL_N $ in \cite{MZ-Voronoi} is normalized (with only an extra factor $1/|y|$ in the Hankel transform) so that it coincides with the classical Vorono\"i summation formula for $\GL_2$ and   Poisson summation formula for $\GL_1$.  }  and incorporate the Bessel kernel of the Hankel integral transform from \cite{Qi-Bessel}\footnote{  In comparison to  \cite{MZ-Voronoi}, the Hankel integral transform  for $\GL_N $ in \cite{Qi-Bessel} differs slightly in the argument by the sign $(-1)^{N}$ so that it is   the inverse Fourier transform in the case of $\GL_1$. Of course there is no difference in our setting of $\GL_4$.}. 

\begin{defn}[Kloosterman sum] \label{defn: Kloosterman}
	 Let $ a      , n \in \BZ$, $c, q_1, q_2, d_1, d_2 \in \BZ_+$ be such that  
	 \begin{align*}
	 	d_1 | c q_1, \qquad d_2 \bigg| \frac {cq_1 q_2} {d_1 }. 
	 \end{align*}
 Define the Kloosterman sum 
	 \begin{align*}
	 	\mathrm{Kl}_2 (a      ,n,c;q_1,q_2,d_1,d_2)=\mathop{\mathop{\sumx\sumx}_{v     _1 (\mathrm{mod}\,
	 			cq_1/d_1 )}}_{v     _2 (\mathrm{mod}\,cq_1q_2/d_1d_2 )} {e}\left(\frac{ a       v     _1 d_1}{c}+\frac{\overline{v     }_1 v     _2 d_2 }{cq_1/d_1}+\frac{n\overline{v     }_2}{cq_1q_2/d_1d_2}\right),   
	 \end{align*}
\end{defn}

\begin{defn}[Hankel transform]  \label{defn: Hankel}
	For $\omega  \in C_c^{\infty} (\BR_+)$ define its Hankel integral transform 
	\begin{align*}
		\Omega (y) = \int_{\BR_{+} }  \omega (x) J_{\phi} (  x y) \nd x, \qquad \text{\rm($y \in \BR_{-} \cup \BR_{+} $)}, 
	\end{align*} 
	with the Bessel kernel $ J_{\phi} (  x  )  $ associated to $\phi$ (as in \cite[\S   3.3]{Qi-Bessel}). 
\end{defn}

\begin{defn}[Bessel kernel]  \label{defn: Bessel}
	More explicitly, 
	\begin{align*}
		J_{\phi} (\pm x) = \frac 1 {4\pi i} \int_{-i\infty}^{ i \infty} \big( G^{0}_{ \phi } (s) \pm G^{1} _{ \phi } (s) \big) x^{-s } \nd s, 
	\end{align*}
	where the curved integral contour passes to the right of $ \lambda_{l} - \mathbf{N}_0 = \{ \lambda_l, \lambda_l - 1, \cdots \} $, and
	\begin{align*}
		G_{\phi}^{\delta} (s) = \prod_{l=1}^{4} G_{\delta} (s - \lambda_l), \qquad G_{\delta} (s) =      \left\{ \begin{aligned}
			&  2 { (2\pi)^{ -s }  }    \Gamma  ( s  )  \cos   (\pi  s/2 )   ,   & & \text{ if } \delta = 0, 
			\\
			&  2i { (2\pi)^{ -s }  }    \Gamma  ( s  )  \sin  (\pi s/2 )  ,  & & \text{ if } \delta = 1 ;  
		\end{aligned} \right. 
	\end{align*}
	here $G_{\delta} (s)$  may also be expressed as the gamma quotient
	\begin{align*}
		G_{\delta} (s) = i^{\delta} \frac {\Gamma_{\BR} (\delta + s) } {\Gamma_{\BR} (\delta + 1-s)},  \qquad \Gamma_{\BR} (s) = \pi^{-s/2} \Gamma (s/2),
	\end{align*}
	so that
	\begin{align*}
		G_{\phi}^{\delta} (s) = \frac { \gamma_{\delta} (s, \widetilde{\phi})  } { \gamma_{\delta} (1-s, \phi)  }, \quad \gamma_{\delta} (s, \phi) = \prod_{l=1}^{4} \Gamma_{\BR} (\delta + s + \lambda_l),  \ \gamma_{\delta} (s, \widetilde{\phi}) = \prod_{l=1}^{4} \Gamma_{\BR} (\delta + s - \lambda_l). 
	\end{align*}
\end{defn}

\begin{lem}[Vorono\"i summation formula]\label{lem: Voronoi} 
Let notation be as above. We have  
\begin{equation*} 
	\begin{split}
		\sum_{n=1}^{\infty}    {A(q_2,q_1   ,n)}  e \Big( & \frac{\widebar{a      }n}{c} \Big)   \omega (n)  = \frac 1 {c^3 q_1^2q_2} \sum_{\pm}   \sum_{d_1\mid cq_1   }\sum_{d_2\mid cq_1q_2/d_1}d_1^2d_2     \\
		& \cdot   \sum_{n =1}^{\infty}   A(n , d_2,d_1)     {\mathrm{Kl}_2 \left({a      }, \mp n, c; q_1,q_2,d_1,d_2\right)}   \Omega \bigg( \! \! \pm \frac{d_1^3d_2^2 n} {c^4q_1^2q_2} \bigg).  
	\end{split}
\end{equation*}  
\end{lem}

Finally, let us summarize the fundamental properties of the Bessel kernel $ J_{\phi} (x) $ in the next lemma. 

\begin{lem}\label{lem: Bessel}
	We have 
	\begin{align}\label{5eq: bound x<1}
		x^{j}	J_{\phi}^{(j)} (x) \Lt_{j, \phi}   \frac 1 {\sqrt{|x|}},
	\end{align}
	for $x \Lt 1$, and for any integer $K > 0$,    
	\begin{align}\label{4eq: asymptotic, Bessel, R} 
		J_{\phi}  ( x ) = \sum_{\pm}    \frac { { e  \left(   \pm  4 x^{1/4} \right)   }}  {x^{3/8} }   \sum_{k= 0}^{K-1} \frac {B^{\pm}_{k } }  {x^{  k/4  }}   +  O_{K, \phi}  \bigg( \frac 1 {x^{ K  / 4}} \bigg), \quad 
		& J_{\phi}  ( -x ) =O_{K, \phi}  \bigg( \frac 1 {x^{ K  / 4}} \bigg),
	\end{align}
	for  $x \Gt 1$. 
\end{lem}
\begin{proof}
	 By the Mellin--Barnes integral for the Bessel kernel $ J_{\phi} (  x  )  $ as above in Definition \ref{defn: Bessel}, along with the Stirling formula, one may readily deduce \eqref{5eq: bound x<1} from $|\mathrm{Re} (\lambda_l)| < 1/2$ (see Remark \ref{rem: Kim-Sarnak}). 
	 The reader is referred to \cite[\S 5.1]{Qi-Wilton} for the proof of $ J_{\phi} (x) \Lt_{  \phi}     1 / {\sqrt{|x|}} $.  
	 On the other hand,  the asymptotic and bound in \eqref{4eq: asymptotic, Bessel, R} are the $\GL_4 (\BR)$ special case of  \cite[Theorem  14.1]{Qi-Bessel}.  
\end{proof}

Moreover, for the derivatives of $ J_{\phi}  ( x )  $ similar asymptotics and bounds  as in \eqref{4eq: asymptotic, Bessel, R}  also hold (see \cite[Theorem 11.24]{Qi-Bessel}). Consequently, for all $x$ one has uniformly (but crudely) 
\begin{align}\label{5eq: x < 1, 2}
	x^{j}	J_{\phi}^{(j)} (x) \Lt_{j, \phi}   \frac {1 + |x|^{(2j+1)/8 }}  {\sqrt{|x|}} . 
\end{align}
Later, in practice, \eqref{4eq: asymptotic, Bessel, R}  will be applied with $x \Gt T^{\vepsilon}$, while \eqref{5eq: x < 1, 2} will be used to show that $ \sqrt{|x|} J_{\phi} (x) $ is just  $T^{\vepsilon}$-inert  for $x \Lt T^{\vepsilon}$ (see Definition \ref{defn: inert}). 

\begin{remark}
	Albeit not as explicit as a kernel function, the asymptotic expansion in {\rm\eqref{4eq: asymptotic, Bessel, R}} is visible in \cite[Lemma 5.2]{Chandee-Li-GL(4)-Special-Points} {\rm(}see also \cite[Lemma 6.1]{XLi} or \cite[Lemma 6]{Blomer}{\rm)}. Note that there are 3 different proofs of the asymptotic formula in \cite{Qi-Bessel}. For a comparison the interested reader is referred to \cite[Appendix B]{Qi-Bessel}.
\end{remark}

\begin{remark}\label{rem: Kim-Sarnak}
	Note that one may slightly improve {\rm\eqref{5eq: bound x<1}} by the Kim--Sarnak bound $ |\mathrm{Re} (\lambda_{l} )| \leqslant 1/2 - 1/11$ 
	{\rm(}see \cite[Appendix 2, Proposition 1]{Kim-Sarnak}{\rm)}. 
\end{remark}

\begin{remark}
	 Note that in the $\GL_3$ setting of \cite{Qi-GL(3)-Special-Points}, however, there is no need to treat the small-argument case $x \Lt T^{\vepsilon}$. 
\end{remark}

\section{Proof of Theorem \ref{thm: main}}

For $ T^{\vepsilon} \leqslant M \leqslant T^{1-\vepsilon}$, our aim is to prove  
\begin{align}\label{5eq: main bound}
	  \sum_{ T < t_j \leqslant T+M }       | L (s_j, \phi \times u_j)   |^2 \Lt_{\vepsilon, \phi}   \frac {T^{5+\vepsilon}} {M^{3}}.
\end{align}
This bound is optimal  when $M = T^{1-\vepsilon}$ and hence we arrive at \eqref{1eq: bound GL(4)} in Theorem \ref{thm: main}. 

\subsection{Initial Reductions}  
The Rankin--Selberg $L$-function $L (s, \phi \times u_j)$ is  defined by 
\begin{align*}
	L (s, \phi \times u_j) = \sum_{n_2    =1}^{\infty} \sum_{n    =1}^{\infty}\frac{A(1, n_2    ,n    )\lambda_j (n    )}{(n_2    ^2 n    )^s},  
\end{align*}
for $\mathrm{Re} (s) > 1$, and by analytic continuation to the entire complex plane.  See  \cite[\S 12.3]{Goldfeld}\footnote{Note that there are some typos in \cite[Definition 12.3.4]{Goldfeld}.}. 

Let $T < t_j \leqslant T+M$. By \cite[Theorem 5.3]{IK}, the approximate functional equation is given by
\begin{align*}
	\begin{aligned}
		L  (  1 / 2 + i t_j , \phi \times u_j  )      = &   \mathop{\sum \sum}_{ n_2  , n     }  \frac{A(1, n_2    ,n    )\lambda_j (n    )}{(n_2^2 n    )^{1/2+it_j} }   V_{\delta_j}    \big( n_2^2 n    ; 1/2+it_j \big)   \\
		& +  \epsilon_j (\phi)    \mathop{\sum \sum}_{ n_2, n     }  \frac{\overline{A(1, n_2    ,n    )}\lambda_j (n    )}{(n_2^2 n    )^{1/2-it_j} }   \widetilde{V}_{\delta_j}    \big( n_2^2 n    ; 1/2-it_j \big) ,  
	\end{aligned}
\end{align*}
where $\delta_j$ is the parity of $u_j (z)$ and $\epsilon_j(\phi)$ has unit norm,
\begin{align*}
	V_{\delta} (y; 1/2+it) =  \frac 1 {2   \pi i }   \int_{ (3)}   & G_{\delta}  (v, it;  \phi )   y^{ - v}   \frac {\nd v} {v}, 
\end{align*}
\begin{align*}
	G_{\delta}  (v, it; \phi   ) =  
	\frac { \gamma_{\delta} (1/2 + 2i t    + v, \phi  )} {  \gamma_{\delta} (1/2 + 2i t    , \phi  )} \frac {  \gamma_{\delta} (1/2     + v, \phi  )} {  \gamma_{\delta} (1/2     , \phi  )} \exp({v^2}) ,
\end{align*} 
with the $\gamma$-factor $\gamma_{\delta} (s, \phi)$ given as in Definition \ref{defn: Bessel}, while $\widetilde{V}_{\delta}(y; 1/2-it)$ is defined similarly with $\phi \ra  \widetilde{\phi}$. By \cite[Proposition 5.4]{IK}, the sums can be effectively restricted to the range $n_2^2 n \leqslant T^{2+\vepsilon}$ at the cost of a negligibly small error. To facilitate our analysis, we use the following expression due to Blomer \cite[Lemma 1]{Blomer} (slightly modified):
\begin{align*}
	V_{\delta} (y; 1/2+it) =  \frac 1 {2   \pi i }\int_{    \vepsilon - i  U }^{\vepsilon + i U } G_{\delta}  (v, it;  \phi )   y^{ - v}   \frac {\nd v} {v}+O_\vepsilon\bigg(\frac{T^\vepsilon}{y^\vepsilon\exp (U^2/2 )}\bigg).
\end{align*}
The error term is negligibly small if we choose $U = \log T$. Note that for any $v$ on the integration contour,   $$G_{\delta}  (v, it; \phi   ) =  O_{\vepsilon,\phi}(T^\vepsilon), $$ by the Stirling  formula, provided that $T < |t| \leqslant T+M$. 

Now, we apply a smooth dyadic partition to the $n$-sum and the Cauchy--Schwarz  inequality to pull out the $v$-integral, and it follows that up to a negligible error 
\begin{align}\label{2eq: AFE} 
	| L  ( s_j , \phi \times u_j  )   |^2 \Lt  {T^{\vepsilon}}     \max_{P \shskip \leqslant  T^{2+\vepsilon}}   \int_{    \vepsilon - i  \log T}^{\vepsilon + i \log T}   \big| S_j^v (P)  \big|^2     \nd v ,
\end{align}
where   $P$ are dyadic,  and  
\begin{align}\label{5eq: S(P)}
	S_j^v (P)  = \frac 1 {\sqrt{P}} \mathop{\sum\sum}_{ n_2,  n     }  \frac{A(1,n_2,n    )\lambda_j (n    )} {(n_2^2 n    )^{ i t_j} } \varww_{ v} \bigg(\frac {n_2^2 n    } {P} \bigg), \quad \ \	\varww_{ v} (x) =   \frac {\varvv (x)}   {x^{1/2+ v} }  , 
\end{align} 
for a certain fixed $\varvv (x) \in C_c^{\infty} [1 , {2}]$.  As $ v = O (\log T)$, it is clear that $\varww_{ v} (x)$ is $\log T$-inert according to Definition \ref{defn: inert}; namely, $ \varww_{v}^{(j)} (x) \Lt_{j} \log^j T$. Further, we use the Cauchy inequality to pull out the $n_2$-sum, obtaining 
\begin{align}\label{5eq: S (N, uj)}
	\big|S_j^v (P)\big|^2  \Lt T^{\vepsilon}  \sum_{n_2 \Lt \sqrt{  P}} \frac 1 {n_2} \bigg|\frac {n_2} {\sqrt{P}} \sum_{n}  {A(1,n_2    , n)\lambda_j (n)}{n^{- i t_j} } \varww_{ v} \bigg(\frac {n} {P /n_2^2}  \bigg)\bigg|^2.  
\end{align} 

By the discussion above, the problem is reduced to proving the following lemma. 

\begin{lem}\label{lem: S}
	For $  N          n_2^2 \Lt T^{2+\vepsilon} $,  define
	\begin{align*}
		S  (n_2; N)    = \sum_{T < t_j \leqslant T+M}     \bigg|\frac { 1    } {\sqrt{N}} \sum_{n}  {A(1,n_2    , n)\lambda_j (n)}{n^{- i t_j} } \varww  \Big(\frac {n} {N}  \Big) \bigg|^2, 
	\end{align*}
	where $ \varww \in C_c^{\infty} [1, 2]$ is $\log T$-inert in the sense of Definition {\ref{defn: inert}}.  
	Then 
	\begin{align*}
		S  (n_2; N)   \Lt     \frac {M T   } {N}  T^{\vepsilon}   \sum_{ n \sim N}   |A (1,n_2, n)|^2   + \bigg(1+ \frac{NT}{M^3}\bigg) N n_2^{2} T^\vepsilon . 
	\end{align*}
\end{lem}

Now the bound in \eqref{5eq: main bound} follows from \eqref{2eq: AFE}, \eqref{5eq: S(P)}, \eqref{5eq: S (N, uj)}, and Lemma \ref{lem: S}, since, by the averaged Ramanujan bound in Lemma \ref{lem: Ramanujan, 1},  for any $P \leqslant T^{2+\vepsilon}$, 
\begin{align*}
\sum_{n_2 \Lt \sqrt{  P}} 	\frac {S \big(n_2; P /n_2^2 \big)} {n_2}  \Lt T^{\vepsilon} \bigg( M T +   P + \frac {P^2 T  } {M^3} \bigg) \Lt  \frac {T^{5+\vepsilon}} {M^{3}} .
\end{align*}

\subsection{Application of Theorem \ref{thm: LS GL(2)}} The rest of this section is devoted to the proof of  Lemma  \ref{lem: S}. First of all, we apply Theorem \ref{thm: LS GL(2)} with
\begin{align*}
	\overline{a}_n  = \frac {1 } {\sqrt{N}}  A(1,n_2, n) \varww  \Big(\frac {n} {N}  \Big),
\end{align*} 
so that, up to a negligible error, we have
\begin{align}
	S  (n_2; N)   \Lt T^{\vepsilon} \big(\breve{D}  (n_2; N) + \breve{P}  (n_2; N)\big) , 
\end{align}
where 
 \begin{align}\label{5eq: P(n; N)}
	\breve{D}  (n_2; N) = \frac {M T   } {N}   \sum_{  n \sim  N} |A (1,n_2, n)|^2  , 
\end{align}
\begin{align} \label{5eq: P(N), 0}
	\breve{P}  (n_2; N)  = \frac{M T} {N}  \sum_{q \Lt N/T}  \frac 1 {q} \int_{-M^{\vepsilon}/ M}^{M^{\vepsilon}/ M}  \sum_{c \Lt N/ T q} \frac 1 {c}  \, \sumx_{   a            (\mathrm{mod} \, c) } \big|  P_{a      } (t/q;  c, n_2; N)  \big|^2   \nd t ,  
\end{align} 
with
\begin{align}\label{5eq: P (t/q)}
	P_{a      } (t/q;  c, n_2; N)  =	 \sum_{ n }  A(1,n_2, n)   e \Big(   \frac {\widebar{a      }     n} {c} \Big)  e \bigg( \frac {n t} {c q} \bigg) \varww  \Big(\frac {n} {N}  \Big) . 
\end{align}
For convenience, let us  truncate  the $t$-integral at $ |t| = 1 / N T $ and then apply a dyadic partition for $ 1/ N T < |t| \leqslant M^{\vepsilon} / M $; by trivial estimation, the resulting error is dominated by $ T^{\vepsilon} \breve{D}  (n_2; N)  $. 
Accordingly, for $ 1/N T \Lt \tau \Lt M^{\vepsilon}/ M$ let us consider
\begin{align} \label{5eq: P(N)}
	\breve{P}_{+}  (\tau; n_2; N)  = \frac{M T} {N}  \sum_{q \Lt N/T}  \frac 1 {q} \int_{\tau }^{2 \tau}   \sum_{c \Lt N/ T q} \frac 1 {c}  \, \sumx_{   a            (\mathrm{mod} \, c) } \big|  P_{a      } (t/q;  c, n_2; N)  \big|^2   \nd t;
\end{align} 
in the same way, we may define and analyze $\breve{P}_{-}  (\tau; n_2; N) $. 

\subsection{Application of the Vorono\"i Summation Formula}  
By applying the Vorono\"i summation formula in Lemma \ref{lem: Voronoi}, the sum $P_{a      } (t/q;  c, n_2; N) $ in \eqref{5eq: P (t/q)} is transformed into 
\begin{equation}\label{5eq: after Voronoi}
	\begin{split}
	\frac 1 {c^3n_2^2}	\sum_{\pm}   \mathop{\sum \sum}_{d_1 d_2 \mid cn_2   }  \!  d_1^2d_2    \sum_{n   }  A(n , d_2,d_1)     {\mathrm{Kl}_2  ({a      }, \mp n, c; n_2,1,d_1,d_2 ) }   \Omega_N^{\pm} \bigg(   \frac{d_1^3d_2^2 n} {c^4n_2^2},\frac{t}{cq}\bigg),
	\end{split}
\end{equation} 
where  
\begin{align}\label{5eq: Omega}
	\Omega_N^{\pm} ( y, r )  = \int J_{\phi} ( \pm x y) e  (   x r  ) \varww  \Big(\frac {x} {N}  \Big) \nd x  . 
\end{align} 
Let us  insert \eqref{5eq: after Voronoi} into \eqref{5eq: P(N)}, drop the $\star$ on the $a$-sum by positivity,   and then pull out the $\pm$- and $(d_1, d_2)$-sums by Cauchy. It follows that $	\breve{P}_{+}  (\tau; n_2; N) $ is bounded by
\begin{align*}
		\frac{M T^{1+\vepsilon}  } {N n_2^4 } \sum_{\pm}  \!  \sum_{q \Lt N/T} \frac1q   \int_{\tau}^{2\tau } \!    \sum_{c \Lt N/ T q}   \!  \frac {1} {c^7}   \!  \mathop{\sum\sum}_{d_1 d_2 \mid c n_2   }  \!   d_1^4d_2^2  \!   \sum_{   a            (\mathrm{mod} \, c) }  \!  \big| \widetilde{P}_a^{\pm} (t/q; c, n_2; d_1, d_2; N) \big|^2   \nd t , 
\end{align*}
if we denote the inner dual $n$-sum in \eqref{5eq: after Voronoi} by $\widetilde{P}_a^{\pm} (t/q; c, n_2; d_1, d_2; N) $.  

For simplicity, let us suppress $\tau$, $n_2$, $N$ from the notation and  consider only the $+$ case; it is much easier to treat the $-$ case, since the Bessel kernel $ J_{\phi} (-x) $ is of rapid decay for $x$ large (see \eqref{4eq: asymptotic, Bessel, R}). So it is left to estimate 
\begin{align}\label{5eq: P++}
	\breve{P}_{\, +}^{+}  =  \frac{M T   } {N n_2^4 }    \sum_{q \Lt N/T} \frac1q   \int_{\tau}^{2\tau } \! \sum_{c \Lt N/ T q}   \frac {1} {c^7}   \!  \mathop{\sum\sum}_{d_1 d_2 \mid c n_2   }   d_1^4d_2^2  \sum_{   a            (\mathrm{mod} \, c) }  \big| \widetilde{P}_a^{+}  (t/q; c  ; d_1, d_2) \big|^2   \nd t, 
\end{align}
with
\begin{align}
	 \widetilde{P}_a^{+}  (t/q; c  ; d_1, d_2) =  \sum_{n   }  A(n , d_2,d_1)     {\mathrm{Kl}_2  ({a      }, - n, c; n_2,   1,d_1,d_2 ) }   \Omega_N^{+} \bigg(  \frac{d_1^3d_2^2 n} {c^4n_2^2};\frac{t}{cq}\bigg) . 
\end{align}

\subsection{Simplification of Exponential Sums}  

After opening the square in the $a$-sum above,  we obtain the exponential sum (see Definition \ref{defn: Kloosterman})
\begin{align*}
	&\sum_{a (\mathrm{mod} \, c) }\mathrm{Kl}_2(a,-m,c;n_2,1,d_1,d_2)
	\overline{\mathrm{Kl}_2(a,-n,c;n_2,1,d_1,d_2)}\\
	=&\sum_{a (\mathrm{mod} \, c) }\mathop{\mathop{\sumx \sumx\sumx\sumx}_{u_1,v_1(\mathrm{mod} \, 
			cn_2/d_1)}}_{u_2,v_2(\mathrm{mod} \, cn_2/d_1d_2)} e \bigg( \frac{ a d_1 (u_1-v_1) }{c} + \frac{u_2\widebar{u}_1-v_2\widebar{v}_1}{cn_2/d_1d_2} - \frac{m\widebar{u}_2-n\widebar{v}_2}{cn_2/d_1d_2}\bigg).
\end{align*}
By orthogonality, the $a$-sum yields the congruence condition $ d_1 (u_1- v_1) \equiv 0 \,(\mathrm{mod}\, c) $, or equivalently $ u_1 \equiv v_1 \,(\mathrm{mod}\, c/ (c, d_1)) $. For brevity, set 
$$   c' = \frac {c} {(c, d_1)}, \qquad n_2' = \frac {n_2} {d_1 / (c, d_1)}.   $$  
Thus we may write $ v_1 = u_1 + c' w $ for $ w\, (\mathrm{mod} \, n_2')$ such that $( u_1 + c' w,  n_2') = 1$, so   the whole $a$-sum in \eqref{5eq: P++} is transformed into
\begin{equation}\label{5eq: sum of S}
	\begin{split}
		c  \sumx_{u_1 (\mathrm{mod}\, c' n_2')}   \mathop{\sum_{w (\mathrm{mod} \, n_2')} }_{(u_1+c'w, n_2') = 1 }  \!\! S_{\widebar{u}_1}^{+}  (t/q; c; d_1,d_2 )  \overline{S_{\,\overline{\!u_1+c' w\!}\,}^{+} (t/q;  c; d_1,d_2) }     ,
	\end{split}
\end{equation} 
where 
\begin{align}   \label{5eq: Sa}
S^{+}_{a}  (t/q; c;  d_1,d_2 ) = \sum_{n }    A(n,d_2 , d_1)  \mathrm{S} (a, - n; c n_2 /d_1 d_2)    \Omega_{N}^{+} \bigg(\frac{d_1^3d_2^2 n} {c^4n_2^2}, \frac {t} {cq} \bigg),
\end{align}
and $\mathrm{S} (m, n; c)$ is the usual Kloosterman sum 
\begin{align*}
	\mathrm{S} (m, n; c) =	\sumx_{v (\mathrm{mod}\, c )} e \bigg(   \frac { m v   +  n \widebar{v} } {c } \bigg). 
\end{align*}
By applying the AM--GM inequality to the $S$-product as in \eqref{5eq: sum of S}, we obtain (half of) the sum of 
\begin{align*}
	c \,  \sumx_{u  (\mathrm{mod}\, c' n_2')}   \mathop{\sum_{w (\mathrm{mod} \, n_2')} }_{(u +c'w, n_2') = 1 }  \big| S_{\widebar{u} }^{+}  (t/q; c; d_1,d_2 ) \big|^2  
\end{align*}
and 
\begin{align*} 
	c \,  \sumx_{u_1 (\mathrm{mod}\, c' n_2')}   \mathop{\sum_{w (\mathrm{mod} \, n_2')} }_{(u_1+c'w, n_2') = 1 }  \big| S_{\, \overline{\! u_1+c' w \!}\, }^{+} (t/q;  c; d_1,d_2) \big|^2 ,
\end{align*} 
whereas, by the change $u = u_1+c'w $, the second sum may be rewritten as 
\begin{align*} 
	c \,  \sumx_{u (\mathrm{mod}\, c' n_2')}   \mathop{\sum_{w (\mathrm{mod} \, n_2')} }_{(u -c'w, n_2') = 1 }  \big| S_{\widebar{u} }^{+}  (t/q; c; d_1,d_2 ) \big|^2 .
\end{align*}
Next, we drop the coprimality conditions $(u \pm c' w, n_2') = 1$, make the substitution $ a = \widebar{u}$, and remove $\star$ on the $a$-sum,  then we may bound \eqref{5eq: sum of S} by  
\begin{align}\label{5eq: a-sum}
	cn_2  \, \sum_{a (\mathrm{mod}\, cn_2/d_1 )}     \big| S_{ a }^{+}  (t/q; c; d_1,d_2 ) \big|^2 .
\end{align} 
Recall here that $ c' n_2' = cn_2/d_1$. 
Similar to the above, if we open the square by \eqref{5eq: Sa}, then the resulting exponential sum reads 
\begin{align*}
\sum_{a (\mathrm{mod}\, c n_2 /d_1  )} & \mathrm{S} (a, - m; c n_2 /d_1 d_2) \overline{\mathrm{S} (a, - n; c n_2 /d_1 d_2)} 	\\
& =  \sum_{a (\mathrm{mod}\, c n_2 /d_1 )}\,\mathop{{\sumx\sumx}}_{u, v (\mathrm{mod}\, c n_2 /d_1 d_2  )} e \bigg(   \frac { a  (u -  v  ) - (\widebar{u}  m -\widebar{v}  n) } {c n_2 /d_1 d_2 }   \bigg) ,
\end{align*}
while, by orthogonality, the $a$-sum yields the congruence $ u \equiv v  \,(\mathrm{mod}\, c n_2 /d_1 d_2) $, so the whole $a$-sum in \eqref{5eq: a-sum} is simplified into 
\begin{align*}
	\frac{c^2n_2^2}{d_1} \, \sumx_{b (\mathrm{mod} \, cn_2/d_1d_2)}    \bigg|  \sum_{n}    A(n,d_2 , d_1)   e \bigg(   \frac {b n } {c n_2/d_1d_2} \bigg)   \Omega_{N}^{+} \bigg(\frac{d_1^3d_2^2 n} {c^4n_2^2}, \frac {t} {cq} \bigg)\bigg|^2.
\end{align*}
In conclusion, for  $ \breve{P}_{\, +}^{+}  $ as in \eqref{5eq: P++}, we have the bound
\begin{align}\label{5eq: bound for P++, 1}
\begin{aligned}
		\breve{P}_{\, +}^{+} \Lt &  \frac{M T } {N n_2^2}     \sum_{q \Lt N/T} \frac1q   \int_{\tau}^{2\tau}\! \sum_{c \Lt N/ T q}  \frac{1}{c^5}  \mathop{\sum\sum}_{d_1d_2| cn_2}   d_1^3d_2^2 \\ & \quad \cdot 
	\sumx_{b (\mathrm{mod} \, cn_2/d_1d_2)}    \bigg| \sum_n    A(n,d_2 , d_1)   e \bigg(   \frac {b n } {c n_2/d_1d_2} \bigg)   \Omega_{N}^{+}\bigg(\frac{d_1^3d_2^2 n} {c^4n_2^2}, \frac {t} {cq} \bigg)\bigg|^2 \nd t . 
\end{aligned}
\end{align}

\subsection{Further Reductions} \label{sec: further reduction} For the analysis of Hankel transform and the application of large sieve, it will be convenient to introduce the new variable  $h = c n_2/d_1d_2$, along with a dyadic partition to the $h$-sum. It suffices to consider 
\begin{align}\label{5eq: P(H)}
	 \breve{P}_{\, +}^{+} (H) = \frac{M T  }  {N   }   \mathop{\sum \sum \sum }_{ d_1d_2q \Lt Nn_2/H T  } \frac {\breve{P}_{\, +}^{+} (d_1, d_2, q; H)} {q} ,   
\end{align}
for  dyadic 
$	H \Lt   {Nn_2} / {T } $,  
where  
\begin{align}
\breve{P}_{\, +}^{+} (d_1, d_2, q; H) = \! \int_{\tau}^{2\tau} \! \sum_{h \sim H}   \frac{1}{c   \gamma} \,   \sumx_{b (\mathrm{mod} \, h)}    \bigg|  \sum_{n}    A(n,d_2, d_1)   e \bigg(   \frac {b n} {h} \bigg)   \Omega_{N}^{+} \bigg(\frac{n} {\gamma}, \frac {t} { c q} \bigg)\bigg|^2\nd t,
\end{align}
\begin{align} \label{5eq: c gamma}
  c = \frac {d_1 d_2 h } {n_2}, \qquad \gamma = \frac { d_1d_2^2 h^4 } { n_2^2 } . 
\end{align}
Moreover, let us set
\begin{align}\label{5eq: Ns}
	N_{\flat} = \frac {T^{\vepsilon} H^4 d_1d_2^2} {Nn_2^2}, \qquad N_{\natural} = \frac {N^3 n_2^2 \tau^4}  {d_1^3 d_2^2 q^4}. 
\end{align}  
For the Hankel transform $\Omega_{N}^{+}  ( {n} / {\gamma},  {t} / { c q} )$ as in \eqref{5eq: Omega}, the  Bessel kernel is oscillatory  if $ n \Gt N_{\flat} $ (so that $ N n / \gamma \Gt T^{\vepsilon}$), in which case Lemma \ref{lem: analysis of integral} (i) will be applied in \S \ref{sec: Hankel II} to 
show that  $\Omega_{N}^{+}  ( {n} / {\gamma},  {t} / { c q} )$ is negligibly small unless $ n \asymp N_{\natural}$. 

\begin{lem}\label{lem: range of Hankel}
	 In the case $ n \Gt N_{\flat} $, the Hankel transform $\Omega_{N}^{+}  ( {n} / {\gamma},  {t} / { c q} )$ is negligibly small unless $ n \asymp N_{\natural}$. 
\end{lem}

Accordingly,   define
\begin{align}
\label{5eq: P flat}	\breve{P}_{\flat} (d_1, d_2, q; H) & = \!  \int_{\tau}^{2\tau} \! \sum_{h \sim H}   \frac{1}{c   \gamma} \,   \sumx_{b (\mathrm{mod} \, h)}   \bigg|    \sum_{n \Lt N_{\flat}}    \! A(n,d_2, d_1)   e \bigg(   \frac {b n} {h} \bigg)   \Omega_{N}^{+} \bigg(\frac{n} {\gamma}, \frac {t} { c q} \bigg)\bigg|^2\nd t, \\
\label{5eq: P natural}	\breve{P}_{\natural} (d_1, d_2, q; H) & = \!  \int_{\tau}^{2\tau} \! \sum_{h \sim H}   \frac{1}{c   \gamma} \,   \sumx_{b (\mathrm{mod} \, h)}    \bigg|   \sum_{n \asymp N_{\natural} }   \!  A(n,d_2, d_1)   e \bigg(   \frac {b n} {h} \bigg)   \Omega_{N}^{+} \bigg(\frac{n} {\gamma}, \frac {t} { c q} \bigg)\bigg|^2\nd t . 
\end{align}
By Cauchy, it is now reduced to proving the following bounds, which will be established at the end of \S \S \ref{sec: large sieve, 1} and \ref{sec: large sieve, 2}. 

\begin{lem}\label{lem: bounds for P}
Let $d_1d_2q \Lt Nn_2/H T$.	 We have 
	  \begin{align}\label{5eq: P flat, bound}
	  	\breve{P}_{\flat} (d_1, d_2, q; H)    \Lt T^{\vepsilon}  \frac {N H} {M}    \frac {n_2} {   d_1^{2/11} d_2^{2/37}   }  , 
  	\end{align}
	  \begin{align}
	  		\label{5eq: P natural, bound}
	  	\breve{P}_{\natural} (d_1, d_2, q; H)    \Lt T^{\vepsilon}   \frac {NH} {M} \frac {   n_2 } {   d_1^{2/11} d_2^{2/37}   } + T^{\vepsilon} \frac {N^3} {M^4} \frac { n_2^2} {  d_1^{24/11} d_2^{39/37} q^3 }    .
	  \end{align}
\end{lem}

By \eqref{5eq: P(H)}, \eqref{5eq: P flat, bound}, and \eqref{5eq: P natural, bound}, we infer that $\breve{P}_{\, +}^{+} (H) $ is bounded by the sum of 
\begin{align*}
	T^{\vepsilon}  \frac {N H  } {M} \frac{M T  }  {N   }  \mathop{\sum \sum \sum }_{ d_1d_2q \Lt Nn_2/H T  }  \frac { n_2 } {   d_1^{2/11} d_2^{2/37} q  } \Lt T^{\vepsilon}  \frac {N H   } {M} \frac{M T  }  {N   } \frac {N n_2^2} {HT} = T^{\vepsilon} N n_2^2 ,
\end{align*}
and
\begin{align*}
	T^{\vepsilon} \frac {N^3} {M^4} \frac{M T  }  {N   }  \mathop{\sum \sum \sum }_{ d_1d_2q \Lt Nn_2/H T  } \frac { n_2^2} {  d_1^{24/11} d_2^{39/37} q^4 } \Lt T^{\vepsilon} \frac {T N^2 n_2^2} {M^3}, 
\end{align*}
as desired in Lemma \ref{lem: S}.  

\subsection{Analysis of the Hankel Transform I: the Case $ \boldsymbol{N y \Lt T^{\vepsilon}}$}\label{sec: Hankel, 1}

Recall from \eqref{5eq: Omega} that 
\begin{align*}
	\Omega_N^{+} ( y, r )  = N \int_1^2   J_{\phi} ( N x y) e  ( N  x r  ) \varww   (x) \nd x  . 
\end{align*}
Let us first consider the easier case when $ N y \Lt T^{\vepsilon}$. 
Recall that $\varww (x)  $ is $\log T$-inert, whereas $ \sqrt{Ny} \,  J_{\phi} ( N x y) $ is $T^{\vepsilon}$-inert in view of \eqref{5eq: x < 1, 2}, so $ \Omega_N^{+} ( y, r )  $ is negligibly small in the case $ N r \Gt T^{\vepsilon} $ (consider it as a Fourier integral). 
Write 
\begin{align}\label{5eq: Omega, flat}
\Omega_N^{+} ( y, r ) = \sqrt{N / y}\,	\varvv_{\flat} (y, r)  . 
\end{align}For $ N y, N r \Lt T^{\vepsilon}$, it follows by trivial estimation that $ \varvv_{\flat} (y, r)  $ is a $T^{\vepsilon}$-inert function as both $ \sqrt{Ny} \,  J_{\phi} ( N x y) $ and $e (Nx r)$ are now $T^{\vepsilon}$-inert. 

\subsection{Application of the Classical Large Sieve}\label{sec: large sieve, 1}

Now we   prove the bound for $\breve{P}_{\flat} (d_1, d_2, q; H) $ in \eqref{5eq: P flat, bound} in Lemma \ref{lem: bounds for P}.  

Note that $ N y \Lt T^{\vepsilon} $ amounts to $ n \Lt N_{\flat} $ for $ y = n / \gamma$ (see \eqref{5eq: c gamma} and \eqref{5eq: Ns}).  Now $\breve{P}_{\flat} (d_1, d_2, q; H)$ in \eqref{5eq: P flat} may be rewritten by \eqref{5eq: Omega, flat} as
\begin{align*}
 {N}	\int_{\tau}^{2\tau} \! \sum_{h \sim H}   \frac{1}{c   } \,   \sumx_{b (\mathrm{mod} \, h)}    \bigg|  \sum_{n \Lt N_{\flat}}    \! \frac {A(n,d_2, d_1)} {\sqrt{n}}   e \bigg(   \frac {b n} {h} \bigg)   \varvv_{\flat} \bigg(\frac{n} {\gamma}, \frac {t} { c q} \bigg)\bigg|^2\nd t . 
\end{align*}
For the  application of the classical large sieve (Lemma \ref{lem: LS}), one may readily dismiss the $T^{\vepsilon}$-inert weight $ \varvv_{\flat}  ( {n} / {\gamma},   {t} / { c q}  )$ by the standard technique using Mellin inversion and Cauchy--Schwarz  
  at the cost of only $T^{\vepsilon}$. To be explicit, if we apply \eqref{2eq: Mellin} (see Remark \ref{rem: Mellin}), then
  \begin{align*}
  	\varvv_{\flat} \bigg(\frac{n} {\gamma}, \frac {t} { c q} \bigg) = \frac 1 {4\pi^2}   \int_{-T^{\vepsilon}  }^{T^{\vepsilon} } \! \int_{-T^{\vepsilon}  }^{T^{\vepsilon} }   \widetilde{\varvv}_{\flat} (i {r}, i s )    \bigg(\frac {n} {\gamma}  \bigg)^{- i  {r}} \bigg(\frac {t} { c q}\bigg)^{-i s}   \nd  {r} \nd s + O_A \big(T^{-A} \big),
  \end{align*}
so, by Cauchy-Schwarz, we infer that, up to a negligible error, 
\begin{align*}
\breve{P}_{\flat} (d_1, d_2, q; H) \Lt 	T^{\vepsilon} N \int_{-T^{\vepsilon}  }^{T^{\vepsilon} } \!	\int_{\tau}^{2\tau} \!  \sum_{h \sim H}   \frac{1}{c   } \,   \sumx_{b (\mathrm{mod} \, h)}    \bigg|  \sum_{n \Lt N_{\flat}}    \! \frac {A(n,d_2, d_1)} {{n}^{1/2+ir} }   e \bigg(   \frac {b n} {h} \bigg)   \bigg|^2\nd t \, \nd r. 
\end{align*}
Note that the bi-variable Mellin inversion is indeed needed since $ h $ is contained in both $\gamma$ and $c$ (see \eqref{5eq: c gamma}), but the $s$-integral is gone after Cauchy-Schwarz since $n$ is not involved. 
By applying Lemma \ref{lem: LS} with $C = O (H)$ and $ N = O (N_{\flat})$,  
\begin{align*}
	 \breve{P}_{\flat} (d_1, d_2, q; H) & \Lt T^{\vepsilon}  {N} \tau  \frac {n_2} {H d_1 d_2   }   \big(H^2 + N_{\flat} \big) \sum_{n \Lt N_{\flat}} \frac {|A(n, d_2, d_1)|^2 } {n} . 
\end{align*}
Recall that $\tau \Lt M^{\vepsilon}/ M$. It follows from $ d_1 d_2 \Lt Nn_2/H T $ and $ N \Lt T^{2+\vepsilon} $ that
\begin{align*}
	N_{\flat} = \frac {T^{\vepsilon} H^4 d_1d_2^2} {Nn_2^2} \Lt \frac {T^{\vepsilon} H^2 N  } {T^2} \Lt T^{\vepsilon} H^2  . 
\end{align*}
Therefore \eqref{5eq: P flat, bound} is now a direct consequence of Lemma \ref{lem: Ramanujan, 2}.

 \subsection{Analysis of the Hankel Transform II:  the Case $ \boldsymbol{N y \Gt T^{\vepsilon}}$} \label{sec: Hankel II}
 Recall from \eqref{5eq: Omega} that 
 \begin{align*} 
 	\Omega_N^{+} ( y, r )  = \int J_{\phi} (   x y) e  (   x r  ) \varww  \Big(\frac {x} {N}  \Big) \nd x  . 
 \end{align*} 
 For the case $ N y \Gt T^{\vepsilon}$, one can apply  the asymptotic expansion  for $J_{\phi} (  x y)$  as in \eqref{4eq: asymptotic, Bessel, R} effectively with a negligibly small error term (choose $K = \lfloor 4 A /\vepsilon \rfloor$ + 1, say). By inserting the asymptotic formula, we infer that  up to a negligible error $ \Omega_N^{+} ( y, r )  $ splits into the sum of 
 \begin{align*}
 	\Omega_N^{+ \pm} ( y, r ) =  \frac {1 } { y^{3/8}}\int   e \big( x r \pm  {4 (x y)^{1/4} }  \big) \varww^{_{\pm}}_{\phi}  \Big(\frac {x} {N}  \Big) \frac {\nd x } { x^{3/8} } , 
 \end{align*}
for  some $\log T$-inert functions  $ \varww^{_{\pm}}_{\phi}  \in C_c^{\infty} [1, 2] $. Next, we make   the change $  x  \ra x y^{1/3}/r^{4/3}$ so that 
\begin{align*} 
	\Omega_{N}^{+\pm} (  y, r ) =  \frac{1}{y^{1/6}r^{5/6}} \int   e \big( (y /r)^{1/3}  (x \pm  4 x^{1/4}   )  \big) \varww^{_{\pm}}_{\phi}  \bigg(\frac {x   } {Nr^{4/3}/y^{1/3}}  \bigg) \frac {\nd x } { x^{3/8} } . 
\end{align*}
By applying Lemma \ref{lem: analysis of integral} with $\gamma =4$,  $\lambda = (y /r)^{1/3}$, $ \rho = Nr^{4/3}/ y^{1/3}$, and $X = \log T$, we infer that $\Omega_{N}^{++} (  y, r )$ is always negligibly small and    $\Omega_{N}^{+-} (  y, r )$ is so except for $ y \asymp N^3 r^4  $ (as claimed in Lemma \ref{lem: range of Hankel}), 
 in which case  
\begin{align} \label{5eq: Omega+-}
	\Omega_{N}^{+-} (y, r ) =  \frac {e    ( - 3 (y /r)^{1/3}  )  \varvv_{\natural}  (  y, r  )  } {  N r^2 }  ,
\end{align} 
for a certain  $\log T$-inert  function $\varvv_{\natural}  ( y, r  ) $.   
 
 \subsection{Application of the Hybrid Large Sieve of Young}\label{sec: large sieve, 2} 
 
 In this sub-section, we verify the bound for
$ \breve{P}_{\natural} (d_1, d_2, q; H)  $ in \eqref{5eq: P natural, bound} and finish the proof of Lemma \ref{lem: bounds for P}. 
 
 Note that the condition $ y \asymp N^3 r^4 $ amounts to $ n \asymp N_{\natural} $ for  $ y = n / \gamma$ (see \eqref{5eq: c gamma} and \eqref{5eq: Ns}). Thus, up to a negligible error,  $\breve{P}_{\natural} (d_1, d_2, q; H)$ in \eqref{5eq: P natural} may be rewritten by \eqref{5eq: Omega+-} as
\begin{align*}
\frac 1 {N^2}	 \int_{\tau}^{2\tau} \! \sum_{h \sim H}   \frac{c^3 q^4}{  \gamma} \,   \sumx_{b (\mathrm{mod} \, h)}    \bigg|  \sum_{n \asymp N_{\natural} }   \!  A(n,d_2, d_1)   e \bigg(   \frac {b n} {h} \bigg) e \bigg(\! - 3 \sqrt[3]{\frac { c q n} {\gamma t }} \bigg)  \varvv_{\natural}  \bigg(\frac{n} {\gamma}, \frac {t} { c q} \bigg)\bigg|^2\frac {\nd t} {t^4} . 
\end{align*}
By the change  $ 1/ \sqrt[3]{t} \ra t$, we obtain the bound  
\begin{align*}
	\frac {1} {N^2 \tau^{8/3}}	\!  \int_{1/\sqrt[3]{2 \tau}}^{1/\sqrt[3]{ \tau}}   \sum_{h \sim H} \!   \frac{c^3 q^4}{  \gamma}   \sumx_{b (\mathrm{mod} \, h)}  \!  \bigg| \!  \sum_{n \asymp N_{\natural} }   \!  A(n,d_2, d_1)   e \bigg(   \frac {b n} {h}  - 3 \sqrt[3]{\frac {  c q} {\gamma   }} \sqrt[3]{n} t \bigg)  \varvv_{\natural}  \bigg(\frac{n} {\gamma}, \frac {1} { c q t^3} \bigg) \bigg|^2 \!  {\nd t}   . 
\end{align*}
Note that by \eqref{5eq: c gamma}
\begin{align*}
	 \frac{c^3 q^4}{  \gamma} = \frac {d_1^2 d_2 q^4} {n_2 } \cdot \frac 1 {h}, \qquad \sqrt[3]{\frac {cq} {\gamma}} = \sqrt[3]{\frac {n_2 q} {d_2} } \cdot \frac 1 {h}. 
\end{align*}
Therefore  an application of   Lemma \ref{lem: Young's LS} with $\gamma = 1/3$,  $\tau \ra 1/ \sqrt[3]{\tau}$, $v =    \sqrt[3]{d_2} / 3 \sqrt[3]{ n_2 q} $, $C = O(H)$, and $N = O (N_{\natural})$ (similar to the application of Lemma  \ref{lem: LS} as in \S \ref{sec: large sieve, 1}, one may as well dismiss the $\log T$-inert weight $ \varvv_{\natural}  ( {n} / {\gamma},   1 / { c q t^3}  )$ here by the standard Mellin technique) yields the estimate
\begin{align*}
\breve{P}_{\natural} (d_1, d_2, q; H) \Lt 	\frac {T^{\vepsilon}} {N^2 \tau^{8/3} } \frac {d_1^2 d_2 q^4} {n_2 }  \bigg( \frac {H} {\tau^{1/3} } +  \bigg({\frac {N_{\natural}^2 d_2}  {n_2 q} }  \bigg)^{1/3}  \bigg)\sum_{n \asymp N_{\natural} }      |A(n,d_2, d_1)|^2. 
\end{align*}
By Lemma   \ref{lem: Ramanujan, 2}, this is further bounded by 
\begin{align*}
	\frac {T^{\vepsilon} N_{\natural} } {N^2 \tau^{8/3} } \frac {d_1^{2+9/11} d_2^{1+35/37} q^4} {n_2 }     \bigg( \frac {H} {\tau^{1/3} } +  \bigg({\frac {N_{\natural}^2 d_2}  {n_2 q} }  \bigg)^{1/3}  \bigg)  . 
\end{align*} 
By the definition in \eqref{5eq: Ns},  $$ N_{\natural} =   \frac {N^3 n_2^2 \tau^4}  {d_1^3 d_2^2 q^4} , $$
we obtain
\begin{align*}
\breve{P}_{\natural} (d_1, d_2, q; H) \Lt  \frac { T^{\vepsilon} NH \tau n_2 } { d_1^{2/11} d_2^{2/37}  } + \frac {T^{\vepsilon} N^3 \tau^4 n_2^2} {d_1^{24/11} d_2^{39/37} q^3 }   . 
\end{align*}
Finally, since $\tau \Lt M^{\vepsilon} / M$,  we arrive at the estimate in \eqref{5eq: P natural, bound}. 
\begin{appendices}
	
	\section{\texorpdfstring{Proof of Lemma \ref{lem: bound for A}: Individual Bound for {\protect\scalebox{1.06}{$A (n_1, n_2, n_3)$}}}{Proof of Lemma \ref{lem: bound for A}: Individual Bound for  A(n\unichar{"2081}, n\unichar{"2082}, n\unichar{"2083})}}
	\label{append}
	
	Actually, we shall prove here for $n_1 n_2 n_3 > 1$ that 
	 \begin{align} 
		|A(n_1, n_2, n_3) |<\tau^3(n_1)\tau^6(n_2)\tau^3(n_3)n_1^{\theta_4}n_2^{\theta_6}n_3^{\theta_4},
	\end{align}
which, by multiplicativity, is reduced for every prime $p$ to 
\begin{align} \label{app: A(p)}  
		|A(p^{\vnu_1},p^{\vnu_2},p^{\vnu_3})|   <(\vnu_1+1)^3(\vnu_2+1)^6(\vnu_3+1)^3p^{\theta_4\vnu_1+ \theta_6\vnu_2 +\theta_4 \vnu_3} . 
\end{align}
To this end, we invoke the Hecke relation in Lemma 3.3 in \cite{Chandee-Li-GL(4)-Special-Points}: 
\begin{align}\label{app: Hecke} 
	\begin{aligned}
	A(p^{\vnu_1}, p^{\vnu_2},   p^{\vnu_3})   =	A(p^{\vnu_1}, 1, 1)A(1, p^{\vnu_2}, p^{\vnu_3}) -   A(p^{\vnu_1-1}, 1, 1)A(1, p^{\vnu_2}, p^{\vnu_3-1})  &  \\  - A(p^{\vnu_1-1}, 1, 1)A(1, p^{\vnu_2-1}, p^{\vnu_3+1}) +   A(p^{\vnu_1-2}, 1, 1) A(1, p^{\vnu_2-1} , p^{\vnu_3})   &;
	\end{aligned}
\end{align}
it is understood that $ A (n_1, n_2, n_3) = 0 $ if one of $n_1$, $n_2$, $n_3$ is not integral. Moreover, keep in mind that $ A (n_3, n_2, n_1) = \overline{ A (n_1, n_2, n_3)} $. 

To start with, we have the bounds of Kim--Sarnak \cite[Appendix 2]{Kim-Sarnak} and Luo--Rudnick--Sarnak   \cite{Luo-Rudnick-Sarnak}\footnote{Kim \cite{Kim-Sarnak} proved that the exterior square $L$-function  $$L (s, \phi, \Lambda^2 ) = \zeta (2s) \sum_{n=1}^{\infty} \frac {A (1, n, 1)} {n^s} $$ is  the $L$-function of a $\GL_6$ automorphic representation. Note that $\zeta (2s)$ is missed in (3.3) in \cite{Chandee-Li-GL(4)-Special-Points}. }:
\begin{align} 
\label{app: Kim-Sarnak}	& 	|A(p^\vnu,1,1)|\leqslant  \binom{\vnu+3}{3}   p^{\theta_4 \vnu} \leqslant \frac{(\vnu+1)^3}{2}p^{\theta_4 \vnu}, \\ 
\label{app: LRS}	& 	|A(1,p^\vnu,1)|\leqslant \binom{\vnu+5}{5} p^{\theta_6 \vnu} + \binom{\vnu+3}{5} p^{\theta_6 (\vnu-2)}   <\frac{(\vnu+1)^5}{4} p^{\theta_6 \vnu}.  
\end{align}
It follows from \eqref{app: Hecke} that
 \begin{align*}
	A(p,p^{\vnu},1)=A(p,1,1)A(1,p^{\vnu},1)-A(1,p^{\vnu-1} , p).
\end{align*}
Thus 
$$|A (p, p, 1)| \leqslant 24 p^{\theta_4 + \theta_6} + 4 p^{\theta_4} < 32 p^{\theta_4 + \theta_6} ,   $$
and, by induction, it is easy to see that 
\begin{align}\label{app: A(p,pv,1)}
	|A(p,p^{\vnu},1)|<\frac{( \vnu + 1)^6}{2}p^{\theta_4+\theta_6 \vnu} ;
\end{align}
indeed, by \eqref{app: Kim-Sarnak}, \eqref{app: LRS}, and   induction hypothesis, it suffices to verify for $\vnu > 1$ that
\begin{align*}
	(\vnu+1)^5 + \frac {\vnu^6} 2 < \frac {(\vnu+1)^6} {2}. 
\end{align*}
Now, by  \eqref{app: Hecke} we have
\begin{align*} 
		A(p^{\vnu_1}, p^{\vnu_2},   1)   =	A(p^{\vnu_1}, 1, 1)A(1, p^{\vnu_2}, 1) & - A(p^{\vnu_1-1}, 1, 1)A(1, p^{\vnu_2-1}, p )   \\  & +   A(p^{\vnu_1-2}, 1, 1) A(1, p^{\vnu_2-1} , 1)     , 
\end{align*}
so it follows from  \eqref{app: Kim-Sarnak}, \eqref{app: LRS}, and \eqref{app: A(p,pv,1)} that 
\begin{align*}
	|A(p^{\vnu_1}, p^{\vnu_2},   1)| < \bigg(\frac{(\vnu_1+1)^3(\vnu_2+1)^5}{8}+\frac{\vnu_1^3\vnu_2^6}{4}+\frac{(\vnu_1-1)^3\vnu_2^5}{8}\bigg)p^{\theta_4 \vnu_1+\theta_6 \vnu_2} , 
\end{align*}
 hence 
\begin{align}\label{app: A(pv,pv,1)}
	|A(p^{\vnu_1}, p^{\vnu_2},   1)| < \frac{ (\vnu_1^3+3 \vnu_1) (\vnu_2+1)^6}{2}p^{\theta_4 \vnu_1+\theta_6 \vnu_2} . 
\end{align}
Finally, \eqref{app: A(p)}   is a direct consequence of \eqref{app: Hecke},   \eqref{app: Kim-Sarnak},  and \eqref{app: A(pv,pv,1)}. 

\end{appendices}
 

\begin{thebibliography}{KPY}
	
	\bibitem[Blo]{Blomer}
	V.~Blomer.
	\newblock Subconvexity for twisted {$L$}-functions on {${\rm GL}(3)$}.
	\newblock {\em Amer. J. Math.}, 134(5):1385--1421, 2012.
	
	\bibitem[CL]{Chandee-Li-GL(4)-Special-Points}
	V.~Chandee and X.~Li.
	\newblock The second moment of {$GL(4)\times GL(2)$} {$L$}-functions at special
	points.
	\newblock {\em Adv. Math.}, 365:107060, 39, 2020.
	
	\bibitem[Gol]{Goldfeld}
D.~Goldfeld.
\newblock {\em Automorphic {F}orms and {$L$}-{F}unctions for the {G}roup {${\rm
			GL}(n, \text{\bf{R}})$}},  {Cambridge Studies in Advanced Mathematics, Vol. 99}.
\newblock Cambridge University Press, Cambridge, 2006.

	\bibitem[IK]{IK}
H.~Iwaniec and E.~Kowalski.
\newblock {\em Analytic {N}umber {T}heory},  {American
	Mathematical Society Colloquium Publications, Vol. 53}.
\newblock American Mathematical Society, Providence, RI, 2004.
	
	\bibitem[IL]{Iwaniec-Li-Ortho}
	H.~Iwaniec and X.~Li.
	\newblock The orthogonality of {H}ecke eigenvalues.
	\newblock {\em Compos. Math.}, 143(3):541--565, 2007.
	
	\bibitem[Iwa]{Iwaniec-L(1)}
	H.~Iwaniec.
	\newblock Small eigenvalues of {L}aplacian for {$\Gamma_0(N)$}.
	\newblock {\em Acta Arith.}, 56(1):65--82, 1990.
	
	\bibitem[Kim]{Kim-Sarnak}
	H.~H. Kim.
	\newblock Functoriality for the exterior square of {${\rm GL}_4$} and the
	symmetric fourth of {${\rm GL}_2$}.
	\newblock {\em J. Amer. Math. Soc.}, 16(1):139--183, 2003.
	\newblock With {A}ppendix 1 by Dinakar Ramakrishnan and {A}ppendix 2 by Kim and
	Peter Sarnak.
	
	\bibitem[KPY]{KPY-Stationary-Phase}
	E.~M. Kiral, I.~Petrow, and M.~P. Young.
	\newblock Oscillatory integrals with uniformity in parameters.
	\newblock {\em J. Th\'eor. Nombres Bordeaux}, 31(1):145--159, 2019.
	
	\bibitem[Li]{XLi}
	X.~Li.
	\newblock The central value of the {R}ankin-{S}elberg {$L$}-functions.
	\newblock {\em Geom. Funct. Anal.}, 18(5):1660--1695, 2009.
	
	\bibitem[LRS]{Luo-Rudnick-Sarnak}
	W.~Luo, Z.~Rudnick, and P.~Sarnak.
	\newblock On {S}elberg's eigenvalue conjecture.
	\newblock {\em Geom. Funct. Anal.}, 5(2):387--401, 1995.
	
	\bibitem[Luo]{Luo-Twisted-LS}
	W.~Luo.
	\newblock The spectral mean value for linear forms in twisted coefficients of
	cusp forms.
	\newblock {\em Acta Arith.}, 70(4):377--391, 1995.
	
	\bibitem[Mon]{Montgomery-Topics}
H.~L. Montgomery.
\newblock {\em Topics in {M}ultiplicative {N}umber {T}heory},    {Lecture Notes in Mathematics, Vol. 227}.
\newblock Springer-Verlag, Berlin-New York, 1971.
	
	\bibitem[MS]{Miller-Schmid-2009}
	S.~D. Miller and W.~Schmid.
\newblock A general {V}oronoi summation formula for {${\rm GL}(n,\mathbb{Z})$}.
\newblock   {\em Geometry and {A}nalysis. {N}o. 2}, Adv. Lect. Math., Vol. 18,
173--224. Int. Press, Somerville, MA, 2011.
	
	\bibitem[MZ]{MZ-Voronoi}
	S.~D. Miller and F.~Zhou.
	\newblock The balanced {V}oronoi formulas for {${\rm GL}(n)$}.
	\newblock {\em Int. Math. Res. Not. IMRN}, (11):3473--3484, 2019.
	
	\bibitem[Qi1]{Qi-Wilton}
	Z.~Qi.
	\newblock Cancellation in the additive twists of {F}ourier coefficients for
	{$\rm GL_2$} and {$\rm GL_3$} over number fields.
	\newblock {\em Amer. J. Math.}, 141(5):1317--1345, 2019.
	
	\bibitem[Qi2]{Qi-Bessel}
	Z.~Qi.
	\newblock Theory of fundamental {B}essel functions of high rank.
	\newblock {\em Mem. Amer. Math. Soc.}, 267\allowbreak(1303):vii+123, 2020.
	
	\bibitem[Qi3]{Qi-GL(3)-Special-Points}
	Z.~Qi.
	\newblock The second moment of {$\mathrm{GL}_3 \times \mathrm{GL}_2$}
	{$L$}-functions at special points.
	\newblock {\em Math. Ann.}, 393(1):\allowbreak1429--1457, 2025.
	
	\bibitem[You]{Young-GL(3)-Special-Points}
	M.~P. Young.
	\newblock The second moment of {$GL(3)\times GL(2)$} {$L$}-functions at special
	points.
	\newblock {\em Math. Ann.}, 356\allowbreak(3):1005--1028, 2013.
	
\end{thebibliography}

\def\cprime{$'$}

\end{document}